\documentclass[12pt]{amsart}
\usepackage[left=2.25cm, right=2.25cm, top=2.6cm, bottom=2.5 cm]{geometry}

\usepackage{hyperref}
\usepackage{enumerate} 	%to be able to set roman numerals in enumerate enviroment: \begin{enumerate}[(i)]
	\usepackage{graphicx}	%\includegraphics and also \colorbox{red}{text}
	
	\usepackage{xcolor} 	%to set color of text: \color{blue}
	\usepackage{tabularx}	
	\usepackage{tcolorbox}	% \begin{tcolorbox} 
		\usepackage{caption}  	%to make captions for drawings without using figure - figure floats!
		\usepackage{tikz}

		\usepackage{kpfonts}  %LADNA CZCIONKA

		\usepackage[T1]{fontenc}	%always goes after font package! NONCOMMUTATIVE
		\usepackage[cp1250]{inputenc}
		\usepackage[polish,english]{babel}
		
		\newtheorem{lem}{Lemma}[section]
		\newtheorem{cor}[lem]{Corollary}
		\newtheorem{defi}[lem]{Definition}
		\newtheorem{thm}[lem]{Theorem}

		\theoremstyle{definition} % te sa pisame \rm
		
		\newtheorem{rk}[lem]{Remark}

		\makeatletter
		\def\th@plain{%
			\thm@notefont{}% same as heading font
			\itshape % body font
		}
		\def\th@definition{%
			\thm@notefont{}% same as heading font
			\normalfont % body font
		}
		\def\th@procedure{
			\thm@notefont{}
			\normalfont
			\addtolength{\leftskip}{2.5em} %indentation, this will indent all of the environment, including the heading
		}
		\makeatother

		\newcommand{\greencomment}[1]{}  %comment not showed
		\newcommand{\forget}[1]{}

		\newcommand{\id}{\textup{id}}

		\usepackage{amssymb}
		\newcommand{\R}{\mathbb{R}}

		\newcommand{\C}{\mathbb{C}}
		\newcommand{\K}{\mathbb{R}}

		\newcommand{\Ps}{\mathbb{P}}
		\newcommand{\Sing}{\mathrm{Sing}}
        \newcommand{\Res}{\mathrm{Res}}

		\author[Z. Jelonek]{Zbigniew Jelonek}
		\author[G. Menani]{Gustavo Menani}
		\author[M. Michalska]{Maria Michalska}
		
		\address{Z. Jelonek, Instytut Matematyczny, Polska Akademia Nauk, \'Sniadeckich 8, 00-656 Warszawa, Poland.}
		
		\email{najelone@cyf-kr.edu.pl}

        \address{G. Menani, Instituto de Ci\^encias Matem\'aticas e de Computa\c{c}\~ao, Universidade de S\~ao Paulo, S\~ao Carlos, SP, Brazil.}

        \email{gmenani@usp.br}

        \address{M. Michalska, Instituto de Ci\^encias Matem\'aticas e de Computa\c{c}\~ao, Universidade de S\~ao Paulo, S\~ao Carlos, SP, Brazil.}
        \email{maria.michalska@wmii.uni.lodz.pl}

         \thanks{The first named author was partially supported by the grant of Narodowe Centrum Nauki, number 2024/55/B/ST1/01412 and by the S\~ao Paulo Research Foundation (FAPESP) under grant number 2025/06706-0. The second named author was supported by FAPESP under grant number 2025/24335-9. The third named author was supported by FAPESP under grant number 2024/04171-9.}

		\title[Lipschitz embeddings of algebraic curves]{\textbf{Lipschitz embeddings of algebraic curves}}
		
\begin{document}

			\maketitle

			\begin{abstract}
				We show that inner Lipschitz classification of real algebraic curves is equivalent to outer Lipschitz classification in $\R^3$ of real algebraic curves with LNE connected components. Moreover, for every real affine algebraic curve we find its LNE model which is real birational to it.  Lastly, we show that this claim does not hold in the complex category.
			\end{abstract}
			
%			\tableofcontents

	%%%%%%%%%%%%
	% THINGS TO REMEMBER BELOW
	%%%%%%%%%%%%		

%			
%			Throughout the paper let $\K=\R$ or $\C$. When we say definable we mean the structure is fixed...
%			
%			\todo{citations to be added}
			
			\greencomment{Let $f:X\to  M$ be a bijection onto image $Y=f(X)$ with $M$ a topological space (resp. manifold). Let  us transplant the   topological (resp.  smooth) structure on $X$ onto $Y$ by means of $f$. Note that $f$ is a topological (resp. smooth) embedding if and only if $\id:Y\to Y$ is a homeomorphism (resp. diffeomorphism) between $Y$ considered with the transplant and $Y$ with structure inherited from the ambient space. Therefore above definition of Lipschitz embedding as an LNE set is natural.}

			\section{Introduction}
			
			There is recently focused interest in metric geometry of algebraic sets, especially Lipschitz classification problems. %, \todo{see for instance \cite{}.}
			Since the result of Birbrair and Mostowski in \cite{birbrairmostowski} has been extended to non-compact case by \cite{AndreVincentMMOnePoint}, one can classify affine  algebraic sets %(real and complex)
			with respect to inner Lipschitz equivalence by means of their LNE models, i.e. sets which are LNE and inner bi-Lipschitz to the initial sets. This means that the inner Lipschitz classification of affine algebraic sets is equivalent to outer Lipschitz classification of a certain class of semialgebraic sets. The tools for outer classification are much stronger and are applied in a restrictive family of sets, thus this approach is very promising.
			
			The essential question is the following: is inner Lipschitz classification of connected algebraic sets equivalent to outer Lipschitz classification of LNE sets which are also algebraic?
			
			In this paper we give the first step for this approach as we answer this question in case of curves. The answer is affirmative in the case of real algebraic curves and negative in the case of complex algebraic curves. In particular, Corollary~\ref{cor:classification} shows that the inner Lipschitz classification of real algebraic curves is equivalent to the outer Lipschitz classification in $\R^3$ of real algebraic curves with LNE connected components.

			\section{Preliminaries}
			
			\subsection{LNE models of sets}

			Throughout this subsection let $(M,d)$ and $(N,d')$ be  metric spaces and let $X\subset M$ and $Y\subset N$.

			\begin{defi}
				A function $f:(M,d)\to (N,d')$ %between metric spaces  
				is said to be \emph{bi-Lipschitz with constant $\dfrac{1}{L}$} if for all $x,y\in M$ we have
				$$
				\dfrac{1}{L}d(x,y)\leq d'(f(x),f(y))\leq Ld(x,y).
				$$	
			\end{defi}
			
			A subset $X$ of a metric space can be equipped with the inner (length) metric~$d_X$ given by infimum on lengths of continuous curves connecting given points as well as with the outer metric which is the restriction of the ambient metric to the set (and so we will simply denote the outer metric the same as the ambient metric). %Note that for connected definable sets the inner metric is a well-defined distance function. %we do not demand that inner metric is well-defined, the following are true without this

			\begin{defi}
				A  set $X\subset M$ is said to be \textit{Lipschitz normally embedded (LNE)} in $M$ if $$\id:(X,d)\to (X,d_X)$$ is bi-Lipschitz.
			\end{defi}
			
			This definition in particular implies that the inner metric on an LNE set $X$ is a well-defined distance function and all points in $X$ can be joined by a curve of finite length.

			\begin{defi}
				We say that $f:X\to Y$ is \emph{outer bi-Lipschitz} if it is bi-Lipschitz with respect to outer metrics on $X$ and $Y$ and we say $f$ is \emph{inner bi-Lipschitz} if it is bi-Lipschitz with respect to inner metrics.
			\end{defi}
			
			Note that  if two sets $X$ and $Y$ are outer bi-Lipschitz, then they are also inner bi-Lipschitz. %Therefore,
			In particular an outer bi-Lipschitz image of an LNE set is LNE.

			\begin{defi}[LNE models]
				Let  $f:X\to N$.
				When $Y:=f(X)$ is LNE and $f:(X,d_X)\to (Y,d_Y)$ is inner bi-Lipschitz, we say that $Y$ is an LNE model of $X$ in $N$.
			\end{defi}

\subsection{Crucial results}

Let us first recall the result of \cite{AndreVincentMMOnePoint} that we will use a few times throughout this paper, independently proved also in \cite{SAMPAIO}.
\begin{thm}[\cite{AndreVincentMMOnePoint}]
	\label{thm:CGMLNE}
	Consider a closed connected definable set in $\R^n$ that does not contain the origin. The following are equivalent:
	\begin{enumerate}
		\item the set is LNE in $\R^n$;
		\item its one-point compactification is LNE in the one-point compactification $\mathbb{S}^n$ of $\R^n$ considered with the natural bounded metric on $\mathbb{S}^n$;
		\item its image by inversion $x\to \frac{x}{\|x\|^2}$ is LNE in $\R^n$.
	\end{enumerate}
\end{thm}

We will also use the crucial method of extending regular  isomorphisms.
\begin{thm}[\cite{Jelonek2008}]\label{thm:JelonekAuto}
Let $X$ be a nonsingular %(not necessarily connected) 
closed subvariety of $\mathbb{R}^n$
of dimension (not necessarily pure) $k$. Let $f = (p_1/q_1, ..., p_n/q_n) : X \to \mathbb{R}^n$ be a regular embedding. If $n \ge 4k + 2$, then there exists a tame regular isomorphism $F =
(P_1/Q_1, ..., P_n/Q_n) : \mathbb{R}^n \to \mathbb{R}^n$ such that $F_{|X} = f.$
\end{thm}

But the most most relevant to the motivation of this paper is the following result of Birbrair and Mostowski \cite{birbrairmostowski} extended to non-compact  case in \cite{AndreVincentMMOnePoint}.

\begin{thm}[\cite{birbrairmostowski,AndreVincentMMOnePoint}]\label{thm:BiMo}
	A closed connected definable set admits a definable LNE model in some $\R^n$.
\end{thm}

\subsection{Semialgebraic LNE models}

In this subsection we show that in the  Birbrair-Mostowski Theorem \ref{thm:BiMo} the LNE models of real algebraic curves can always be embedded into  $\R^3$ and that the obstruction to existence of semialgebraic LNE models on the real plane is purely topological - we express it in means of Kuratowski theorem on graph embeddings.

Throughout this section we assume that $X$ is a definable closed curve in $\R^n$.

\begin{defi}
	For each $p\in X$ and $r>0$ small enough there is a number $k$ such that either $X\cap B(p,r) $ is 
	a $C^1$-embedded submanifold of $\R^n$ and then $k=1$ or 
	$$X\cap B(p,r) \setminus \{p\}=\bigcup_{i=1}^kY_i, $$
	where each branch $Y_i$ is a definable smooth one-dimensional arc which is LNE and $Y_i\cap Y_j=\emptyset$ for all $j\neq i$.
	%Let $X$ be a closed definable curve in $\R^n$. 
	We call the number $k$ the \emph{strata multiplicity} of $p$ and denote by $\mathrm{sm}(p)$.
	%\end{defi}
	%
	%\begin{defi}
	
	The \emph{singularities} of a closed definable curve $X$ denoted $\mathrm{Sing}(X)$ are points $p\in X$ such that $\mathrm{sm}(p)\geq 2$. We say that \emph{a point is of self-intersection} if $\mathrm{sm}(p)>2$.
\end{defi}

Note that if $X$ is a reduced algebraic curve, then $\mathrm{Sing}(X)$ coincides with the standard singular set of the curve and self-intersection points are the points where it is locally reducible.

\begin{defi}
	For a set $X\subset\R^n$ its  \emph{linear tangent cone $L_pX$}  at $p$ is the set of vectors $v\in\K^n$ that are limits of the secant lines passing through the point $p$, i.e. 
	\begin{equation*}
		\lim_{i\to\infty}\dfrac{p_i-p}{t_i}=v
	\end{equation*}
	where $p_i\to p, p_i\in X$ and $t_i\to 0, t_i\neq0$.
\end{defi}
Therefore the linear tangent cone is the smallest flag of linear subspaces containing the standard tangent cone.

\begin{lem} \label{lem:curvesLNEBALLS}
	 Let $p\in X$ and $r>0$ small enough. 	
	There exists a definable inner-bi-Lipschitz mapping 
	$$ 
	f_p: X\cap B(p,r) \to \mathrm{Cone}_p(X\cap S^{n-1}(p,r))\cap B(p,r).
	$$
\end{lem}

\begin{proof} 	
	Let $Z_i:=\overline{Y_i}$ and $y_i$ be the unique point of $Z_i\cap S^{n-1}(p,r)$. 	
	For some $r>0$ sufficiently small, the smooth curve $Z_i\cap B(p,r)$ is inner bi-Lipschitz equivalent %(with respect to the inner metric) 
	to a line segment. By reparametrizing we obtain that  $Z_i$ is inner bi-Lipschitz equivalent via a definable mapping~$f_i$ to the line segment $[p,y_i]$ with Lipschitz constant $L_i$ for all $1\leq i\leq\mathrm{sm}(p)$.

	Define the map $f_p$ from $X\cap{B(p,r)}$ to the union of half lines 
	$$U:=\bigcup_{i=1}^{\mathrm{sm}(p)}[p,y_i)$$
	as $f_p(x)=f_i(x)$ if $x\in Z_i$. Moreover, it is inner bi-Lipschitz  with constant $M=\max_iL_i$. 	
	Indeed,  let $x\in Y_i$, $y\in Y_j$ and $d_i$ be the inner distance of $Y_i$, then, for $i\neq j$,
	\begin{align*}
		d_U(f_p(x),f_p(y))&=d_U(f_i(p),f_i(x))+d_U(f_j(p),f_j(y))\\
		&\leq L_id_i(p,x)+L_jd_j(p,y) \leq Md_X(x,y).
		%	&\leq Md_i(p,x)+Md_j(p,y)\\
		%	&=Md_X(x,y).
	\end{align*}
	Similarly, $d_U(f_p(x),f_p(y))\geq \dfrac{1}{M}d_X(x,y)$.
	
\end{proof}

%
%\

\begin{lem}\label{lemSAMEDIM}
	Any  closed connected definable curve of $\R^n$ has a semialgebraic LNE model in  $\R^n$.
\end{lem}	

\begin{proof}

	Without loss of generality assume origin does not belong to the  closed connected definable curve $X\subset \R^n$. 
	
	Let $Z$ be the one-point closure of the image by inversion  of $X$ in $\R^n$. It is a compact definable curve.  We can choose a $r>0$ that satisfies Lemma~\ref{lem:curvesLNEBALLS} for all the singular points $p$ of~$Z$ and such that $B(p,r)$ are pairwise disjoint. 
	
	Let us define $f:Z\to\R^n$ as $f_p$ of Lemma~\ref{lem:curvesLNEBALLS} on $Z\cap B(p,r)$ whenever $\mathrm{sm}(p)>1$. The function~$f$ takes any connected component of $Z\setminus \bigcup_{p: \mathrm{sm}(p)>1} B(p,r)$ to a broken line with a finite number of segments and fixed end-points so that no two of these broken lines intersect. It is easy to see it can be done inner (and outer) bi-Lipschitzly simply by the fact that these connected components are disjoint compact smooth arcs.  By translation we may assume $f(0)=0$. Then the function $f$ is clearly inner-bi-Lipschitz and definable.

	The image $f(Z)$ in $\R^n$  is a semialgebraic curve. It is clearly LNE and by Theorem~\ref{thm:CGMLNE} its image by inversion is LNE. hus the latter, by construction, is an LNE model of~$X$. 
\end{proof}

A graph in $\R^n$ is a finite collection of vertices (points) and edges (arcs) in $\R^n$ between vertices such that the interior of an edge contains no vertex and no point of any other edge. %, see for instance \cite{AGRAPHBOOK}. 
An abstract graph is a collection of vertices and edges that have both end-points on vertices. An embedding of a graph is a graph isomorphism between an abstract graph and a graph in $\R^n$, see~\cite{Mohar2001}. Note that any connected finite graph can be embedded as an LNE semialgebraic curve.

Denote by $\overline{X}^{\mathbb{S}^n}$ the closure of $X$ in the one-point compactification $\mathbb{S}^n$ of $\R^n$. 
\begin{defi}\label{defGRAPH}
	The \textit{abstract compact graph} of a closed definable curve $X$ is the pair $G(X)=(V(X),E(X))$ where
	\begin{itemize}
		\item The set $V(X)$ of vertices consists of all points $p\in\mathrm{Sing}(\overline{X}^{\mathbb{S}^n})$ such that $\mathrm{sm}(p)\geq 2$;
		\item The set $E(X)$ of edges consists of all $Y\subset X$ that are homeomorphic to $(0,1)$, disjoint from $V(X)$ and such that $\overline{Y}\backslash Y\subset V(X)$%, the latter set being the vertices associated to this edge.
	\end{itemize}
\end{defi}

For closed definable curves the Definition~\ref{defGRAPH} of an abstract compact graph  is well-posed and 
the degree of a vertex of the graph %in abstract  graph 
is equal to strata multiplicity of the curve:
$$\deg p = \mathrm{sm}(p).$$
%\end{prop}

Kuratowski graphs are two specific graphs: the complete bipartite graph on two sets of $3$ vertices each $K_{3,3}$ and the complete graph on $5$ vertices $K_5$.

\tikzset{every picture/.style={line width=0.75pt}} %set default line width to 0.75pt        
\begin{center}
	\begin{tikzpicture}[x=0.6pt,y=0.6pt,yscale=-1,xscale=1]
		%uncomment if require: \path (0,472); %set diagram left start at 0, and has height of 472
		
		%Straight Lines [id:da7758371225891249] 
		\draw    (182,104) -- (338,104) ;
		%Straight Lines [id:da8760821356587635] 
		\draw    (182,182) -- (338,182) ;
		%Straight Lines [id:da7541401273718503] 
		\draw    (182,260) -- (338,260) ;
		%Straight Lines [id:da7605305779644144] 
		\draw    (182,104) -- (338,260) ;
		\draw [shift={(338,260)}, rotate = 45] [color={rgb, 255:red, 0; green, 0; blue, 0 }  ][fill={rgb, 255:red, 0; green, 0; blue, 0 }  ][line width=0.75]      (0, 0) circle [x radius= 3.35, y radius= 3.35]   ;
		\draw [shift={(182,104)}, rotate = 45] [color={rgb, 255:red, 0; green, 0; blue, 0 }  ][fill={rgb, 255:red, 0; green, 0; blue, 0 }  ][line width=0.75]      (0, 0) circle [x radius= 3.35, y radius= 3.35]   ;
		%Straight Lines [id:da3457516856694647] 
		\draw    (182,260) -- (338,104) ;
		\draw [shift={(338,104)}, rotate = 315] [color={rgb, 255:red, 0; green, 0; blue, 0 }  ][fill={rgb, 255:red, 0; green, 0; blue, 0 }  ][line width=0.75]      (0, 0) circle [x radius= 3.35, y radius= 3.35]   ;
		\draw [shift={(182,260)}, rotate = 315] [color={rgb, 255:red, 0; green, 0; blue, 0 }  ][fill={rgb, 255:red, 0; green, 0; blue, 0 }  ][line width=0.75]      (0, 0) circle [x radius= 3.35, y radius= 3.35]   ;
		%Straight Lines [id:da833113468935486] 
		\draw    (182,104) -- (338,182) ;
		%Straight Lines [id:da012788123690497422] 
		\draw    (182,260) -- (338,182) ;
		\draw [shift={(338,182)}, rotate = 333.43] [color={rgb, 255:red, 0; green, 0; blue, 0 }  ][fill={rgb, 255:red, 0; green, 0; blue, 0 }  ][line width=0.75]      (0, 0) circle [x radius= 3.35, y radius= 3.35]   ;
		\draw [shift={(182,260)}, rotate = 333.43] [color={rgb, 255:red, 0; green, 0; blue, 0 }  ][fill={rgb, 255:red, 0; green, 0; blue, 0 }  ][line width=0.75]      (0, 0) circle [x radius= 3.35, y radius= 3.35]   ;
		%Straight Lines [id:da16233539560628774] 
		\draw    (182,182) -- (338,104) ;
		\draw [shift={(338,104)}, rotate = 333.43] [color={rgb, 255:red, 0; green, 0; blue, 0 }  ][fill={rgb, 255:red, 0; green, 0; blue, 0 }  ][line width=0.75]      (0, 0) circle [x radius= 3.35, y radius= 3.35]   ;
		\draw [shift={(182,182)}, rotate = 333.43] [color={rgb, 255:red, 0; green, 0; blue, 0 }  ][fill={rgb, 255:red, 0; green, 0; blue, 0 }  ][line width=0.75]      (0, 0) circle [x radius= 3.35, y radius= 3.35]   ;
		%Straight Lines [id:da8749843785594066] 
		\draw    (338,260) -- (182,182) ;
		\draw [shift={(182,182)}, rotate = 206.57] [color={rgb, 255:red, 0; green, 0; blue, 0 }  ][fill={rgb, 255:red, 0; green, 0; blue, 0 }  ][line width=0.75]      (0, 0) circle [x radius= 3.35, y radius= 3.35]   ;
		\draw [shift={(338,260)}, rotate = 206.57] [color={rgb, 255:red, 0; green, 0; blue, 0 }  ][fill={rgb, 255:red, 0; green, 0; blue, 0 }  ][line width=0.75]      (0, 0) circle [x radius= 3.35, y radius= 3.35]   ;
		%Straight Lines [id:da24098584603315876] 
		\draw    (546,104) -- (468,156) ;
		\draw [shift={(468,156)}, rotate = 146.31] [color={rgb, 255:red, 0; green, 0; blue, 0 }  ][fill={rgb, 255:red, 0; green, 0; blue, 0 }  ][line width=0.75]      (0, 0) circle [x radius= 3.35, y radius= 3.35]   ;
		\draw [shift={(546,104)}, rotate = 146.31] [color={rgb, 255:red, 0; green, 0; blue, 0 }  ][fill={rgb, 255:red, 0; green, 0; blue, 0 }  ][line width=0.75]      (0, 0) circle [x radius= 3.35, y radius= 3.35]   ;
		%Straight Lines [id:da43400198842111803] 
		\draw    (494,260) -- (468,156) ;
		\draw [shift={(468,156)}, rotate = 255.96] [color={rgb, 255:red, 0; green, 0; blue, 0 }  ][fill={rgb, 255:red, 0; green, 0; blue, 0 }  ][line width=0.75]      (0, 0) circle [x radius= 3.35, y radius= 3.35]   ;
		\draw [shift={(494,260)}, rotate = 255.96] [color={rgb, 255:red, 0; green, 0; blue, 0 }  ][fill={rgb, 255:red, 0; green, 0; blue, 0 }  ][line width=0.75]      (0, 0) circle [x radius= 3.35, y radius= 3.35]   ;
		%Straight Lines [id:da6377675467800552] 
		\draw    (624,156) -- (546,104) ;
		\draw [shift={(546,104)}, rotate = 213.69] [color={rgb, 255:red, 0; green, 0; blue, 0 }  ][fill={rgb, 255:red, 0; green, 0; blue, 0 }  ][line width=0.75]      (0, 0) circle [x radius= 3.35, y radius= 3.35]   ;
		\draw [shift={(624,156)}, rotate = 213.69] [color={rgb, 255:red, 0; green, 0; blue, 0 }  ][fill={rgb, 255:red, 0; green, 0; blue, 0 }  ][line width=0.75]      (0, 0) circle [x radius= 3.35, y radius= 3.35]   ;
		%Straight Lines [id:da9367512630141357] 
		\draw    (624,156) -- (598,260) ;
		\draw [shift={(598,260)}, rotate = 104.04] [color={rgb, 255:red, 0; green, 0; blue, 0 }  ][fill={rgb, 255:red, 0; green, 0; blue, 0 }  ][line width=0.75]      (0, 0) circle [x radius= 3.35, y radius= 3.35]   ;
		\draw [shift={(624,156)}, rotate = 104.04] [color={rgb, 255:red, 0; green, 0; blue, 0 }  ][fill={rgb, 255:red, 0; green, 0; blue, 0 }  ][line width=0.75]      (0, 0) circle [x radius= 3.35, y radius= 3.35]   ;
		%Straight Lines [id:da3121188894973973] 
		\draw    (598,260) -- (494,260) ;
		\draw [shift={(494,260)}, rotate = 180] [color={rgb, 255:red, 0; green, 0; blue, 0 }  ][fill={rgb, 255:red, 0; green, 0; blue, 0 }  ][line width=0.75]      (0, 0) circle [x radius= 3.35, y radius= 3.35]   ;
		\draw [shift={(598,260)}, rotate = 180] [color={rgb, 255:red, 0; green, 0; blue, 0 }  ][fill={rgb, 255:red, 0; green, 0; blue, 0 }  ][line width=0.75]      (0, 0) circle [x radius= 3.35, y radius= 3.35]   ;
		%Straight Lines [id:da5401047665763348] 
		\draw    (546,104) -- (494,260) ;
		\draw [shift={(494,260)}, rotate = 108.43] [color={rgb, 255:red, 0; green, 0; blue, 0 }  ][fill={rgb, 255:red, 0; green, 0; blue, 0 }  ][line width=0.75]      (0, 0) circle [x radius= 3.35, y radius= 3.35]   ;
		\draw [shift={(546,104)}, rotate = 108.43] [color={rgb, 255:red, 0; green, 0; blue, 0 }  ][fill={rgb, 255:red, 0; green, 0; blue, 0 }  ][line width=0.75]      (0, 0) circle [x radius= 3.35, y radius= 3.35]   ;
		%Straight Lines [id:da7175565297127539] 
		\draw    (624,156) -- (494,260) ;
		\draw [shift={(494,260)}, rotate = 141.34] [color={rgb, 255:red, 0; green, 0; blue, 0 }  ][fill={rgb, 255:red, 0; green, 0; blue, 0 }  ][line width=0.75]      (0, 0) circle [x radius= 3.35, y radius= 3.35]   ;
		\draw [shift={(624,156)}, rotate = 141.34] [color={rgb, 255:red, 0; green, 0; blue, 0 }  ][fill={rgb, 255:red, 0; green, 0; blue, 0 }  ][line width=0.75]      (0, 0) circle [x radius= 3.35, y radius= 3.35]   ;
		%Straight Lines [id:da31260948532636756] 
		\draw    (624,156) -- (468,156) ;
		\draw [shift={(468,156)}, rotate = 180] [color={rgb, 255:red, 0; green, 0; blue, 0 }  ][fill={rgb, 255:red, 0; green, 0; blue, 0 }  ][line width=0.75]      (0, 0) circle [x radius= 3.35, y radius= 3.35]   ;
		\draw [shift={(624,156)}, rotate = 180] [color={rgb, 255:red, 0; green, 0; blue, 0 }  ][fill={rgb, 255:red, 0; green, 0; blue, 0 }  ][line width=0.75]      (0, 0) circle [x radius= 3.35, y radius= 3.35]   ;
		%Straight Lines [id:da7483458929437976] 
		\draw    (598,260) -- (468,156) ;
		\draw [shift={(468,156)}, rotate = 218.66] [color={rgb, 255:red, 0; green, 0; blue, 0 }  ][fill={rgb, 255:red, 0; green, 0; blue, 0 }  ][line width=0.75]      (0, 0) circle [x radius= 3.35, y radius= 3.35]   ;
		\draw [shift={(598,260)}, rotate = 218.66] [color={rgb, 255:red, 0; green, 0; blue, 0 }  ][fill={rgb, 255:red, 0; green, 0; blue, 0 }  ][line width=0.75]      (0, 0) circle [x radius= 3.35, y radius= 3.35]   ;
		%Straight Lines [id:da5964723188202141] 
		\draw    (598,260) -- (546,104) ;
		\draw [shift={(546,104)}, rotate = 251.57] [color={rgb, 255:red, 0; green, 0; blue, 0 }  ][fill={rgb, 255:red, 0; green, 0; blue, 0 }  ][line width=0.75]      (0, 0) circle [x radius= 3.35, y radius= 3.35]   ;
		\draw [shift={(598,260)}, rotate = 251.57] [color={rgb, 255:red, 0; green, 0; blue, 0 }  ][fill={rgb, 255:red, 0; green, 0; blue, 0 }  ][line width=0.75]      (0, 0) circle [x radius= 3.35, y radius= 3.35]   ;

	\end{tikzpicture}
	
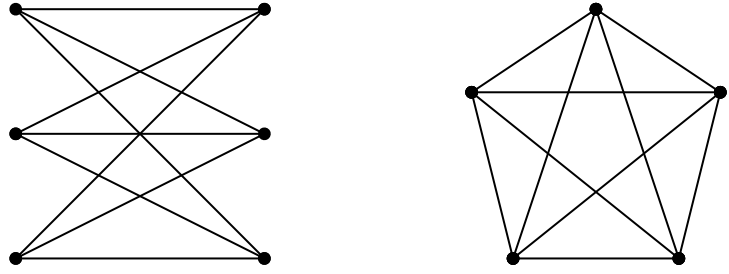
\captionof{figure}{The Kuratowski graphs $K_{3,3}$ and $K_5$.}
\end{center}

Now we can pose the inner Lipschitz classification problem in terms of graphs as follows.

\begin{thm}\label{thm:graphs}
	\ 
	\begin{enumerate}
		\item Any closed connected definable curve has a semialgebraic LNE model in~$\R^3$.
		\item A closed connected definable curve has an LNE model in $\R^2$ if and only if its abstract compact graph does not contain a graph that is a subdivision of a Kuratowski graph as a subgraph.
		\item Any definable curve has a semialgebraic LNE model on a smooth affine algebraic LNE surface.
	\end{enumerate}
\end{thm}

\begin{proof}%[Proof of Theorem \ref{thm:graphs}]

	Recall that by classic Kuratowski's theorem of \cite{Kuratowski} 
	every compact graph can be embedded in~$\R^3$ and a graph can be embedded into~$\R^2$ if and only if it does not contain a graph that is a subdivision of a Kuratowski graph as subgraph.

	Let $G\subset\R^d$ be the embedding of the compact abstract graph of $X$ as an LNE set. Moreover, any graph can be embedded into an $n$-torus, a compact surface with genus~$n$ in $\R^3$, see~\cite{GRAPHS}, which can be taken to be a non-singular real algebraic set. This surface is smooth and compact, therefore LNE. The embedded graph can be considered an LNE semialgebraic subset of this surface, by using for instance Nash approximation~\cite{Shiota}.

	If $X$ is compact, then the embedding $G$ of its compact abstract graph is an LNE model of~$X$. If $X$ is not compact, we may assume that $0\in\R^d$ is the vertex of $G$ corresponding to the point at infinity. Then, the image of $G$ by inversion is an LNE model of $X$ by Theorem~\ref{thm:CGMLNE}. Note that we can take $d=2$ when the graph does not contain a Kuratowski graph and $d=3$ otherwise. Therefore, points (1-2) now follow from Lemma~\ref{lemSAMEDIM}. Moreover, the image by inversion of the algebraic smooth LNE surface is again LNE and algebraic by Theorem~\ref{thm:CGMLNE} and Lemma~\ref{lemma:inversion} which applied to the $n$-torus gives point~(3).
		
\end{proof}

Therefore, straight from the characterization of Theorem~\ref{thm:graphs} we get a following Corollary (note that the type of singularities of the curve is irrelevant).
\begin{cor}
	Any compact algebraic curve with at most $4$ self-intersection points and any algebraic curve with at most $3$ self-intersection points has an LNE model in $\R^2$.  
\end{cor}

			\section{Real algebraic LNE models}\label{sec:algebraicModels}
			
			As we saw in previous section, any algebraic set has a semialgebraic LNE model. 
			Moreover, a generic algebraic singularity or a generic algebraic set by \cite{CoGrMiConic} is  an LNE model of itself. 
			 Question from \cite{AndreVincentMMOnePoint} is: do all  algebraic sets admit real algebraic LNE models? To our mild surprise the answer for real  algebraic curves is affirmative.

			\begin{thm}\label{thm:algebraicModels}
				Any connected real  algebraic curve has a real algebraic LNE model in $\mathbb{R}^3$, which is real birationally equivalent to it. %If the curve is irreducible, then the model can also be taken to be an irreducible curve.
			\end{thm}

	The proof of Theorem~\ref{thm:algebraicModels} is constructive: proceeds by specializing to compact case by using inversion of Lemma~\ref{lemma:inversion}, resolving curve singularities, inductively reconstructing intersections via projections so that all branches meet transversely by Lemma~\ref{lem:transonly} and then coming back by inversion to non-compact case using Theorem~\ref{thm:CGMLNE}. Note that the construction of a convenient map of Lemmas~\ref{lem:transintersections} and \ref{lem:transonly} can be simplified by choosing a regular automorphism with prescribed derivatives at points developed in an upcoming paper \cite{JMM}.

	\begin{defi}
		Let $X\subset\R^n$ be an algebraic curve and  $(X_1,p),\ldots,(X_s,p)$ be the irreducible components of germ $(X,p)$ at $p\in X$ with $s>1$. We say that $X$ is a \emph{transverse intersection at $p$} if for every $1\leq i\neq j\leq s$ we have
		$$L_pX_i\cap L_pX_j=\{0\}.$$
	\end{defi}
	
	Thus a real algebraic curve $X\subset\R^n$ is a transverse intersection at a point $p\in X$ if and only if its linear tangent cone $L_pX$ consists of a finite number of lines through the origin equal to the number of branches of its germ at the point $p$ and greater than $1$.
	
	First, note that a (real or complex) LNE algebraic curve has only smooth points or singularities which are transverse intersections of smooth branches, see for instance~\cite{AlexLevCurves,Birbrair2018,NeumannPichon,DenkowskiTibar}. Globally one needs to also consider asymptotics at infinity, see \cite{MVAlnecurves,CoGrMiConic}. In short, due to aforementioned results and \cite{AndreVincentMMOnePoint}, a real algebraic curve is LNE if and only if its compact image by inversion has singularities which are at most transverse intersections of smooth branches. The next lemma is obvious.

		\begin{lem}\label{lem:transvcone}
			Let $X\subset\R^n$ be an algebraic curve that is a transverse intersection at $p\in X$. If $F:\R^n\to\R^N$ is a regular embedding, then $F(X)$ is a transverse intersection at $F(p)$.
		\end{lem}

				\begin{lem}\label{lem:interpolation}
				    Let $\{p_1,\dots,p_k\}\subset\R^N$ be a set of $k$ distinct points in $\R^N$ and $\{t_1,\dots,t_k\}\subset\R$. There exists a polynomial function $h:\R^N\to\R$ such that $h(p_j)=t_j$ for every $j=1,\dots,k$. 
				\end{lem}

                \begin{proof}
                    For each $j=1,\ldots,k$ define
                    $$L_j(x)=\prod_{\substack{i=1 \\ i\neq j}}^{k}\dfrac{||x-p_i||^2}{||p_j-p_i||^2}.$$

                    Note that $L_j(p_j)=1$ and $L_j(p_s)=0$ if $s\neq j$. Then, the polynomial
                    $$h(x)=\sum_{j=1}^kt_jL_j(x)$$
                    satisfies the $h(p_j)=t_j$ for every $j=1,\ldots,k$.                 
                \end{proof}
            
            \begin{lem}\label{lem:interpolationgradient}
            	Let $S\subset\R^N$ be a finite set of points and $q\in S$. 
            There exists  a polynomial $h:\R^N\to\R$ such that
            \begin{itemize}
            	\item $h(a)=0$ for every $a\in S\setminus\{q\}$;
            	\item $h(q)=1$;
            	\item $\nabla h(a)=0$ for every $a\in S$.
            \end{itemize}
        \end{lem}
        \begin{proof}        
            By Lemma \ref{lem:interpolation} there exists $\tilde h:\R^N\to\R$ such that $\tilde h(a)=0$ for $a\in S\setminus\{q\}$ and $\tilde h(q)=1$. Now define
            $$h(x)=\tilde h(x)^2(1-2\nabla\tilde h(q)\cdot(x-q))$$
            
            Note that $h(q)=1$ and $h(x)=0$ whenever $\tilde h(x)=0$. Furthermore,
            $$\nabla h(x)=-2\tilde h(x)^2\nabla\tilde h(q)+2\tilde h(x)(1-2\nabla\tilde h(q)\cdot(x-q))\nabla\tilde h(x).$$
            
            It follows that $\nabla h(a)=0$ for $a\in S$.
        \end{proof}

				\begin{lem}[Pointwise transverse glueing]\label{lem:transintersections}
					Let $Y$ be a compact algebraic curve in $\R^N$ such that all its singular points %$p_1,\dots,p_{i}$ 
					are  transverse intersections of smooth branches. Let $p,q\in Y$ and $S\subset Y$ be any finite set containing $\Sing(Y)\cup\{p,q\}$. There exists a mapping $\pi:Y\to Z$, where $Z\subset\R^{N'}$ is a compact algebraic curve, such that:
					\begin{enumerate}
                        \item $\pi$ is birational and proper onto its image
                        \item $\pi(Y)$ is Zariski dense in $Z$
						\item $\pi^{-1}(\pi(p))$ consists of two points $p$ and $q$
						%\item curve $\pi(Y)$ is transverse intersection at $\pi(p_i)$
						%\item
						and  $\pi$ is one-to-one on $S\setminus \{p,q\}$, i.e. % for all $p_1,\dots,p_{i-1}$
							$|S|-1=|\pi(S)|$
						\item the curve $\pi(Y)$ has transverse intersections of smooth branches at all points $\pi(\Sing (Y))\cup \pi(p)$ and is smooth at points in $\pi(S\setminus(\Sing(Y)\cup\{p,q\}) )$
					\end{enumerate}	
				
				\end{lem}
			
				\begin{proof}
		Without loss of generality, up to affine translation,  we may suppose $q=0\in\R^N$. Let $Y_1,\ldots,Y_\nu$ be the irreducible components of $Y$. Choose points $\{r_1,\ldots,r_{2\nu}\}\subset Y\setminus S$ such that $r_{2j-1}, r_{2j}$ are two distinct smooth points in the irreducible component $Y_j$ for each $j=1,\ldots,\nu$ .
		
		The construction of $\pi$ will take several steps, it will be defined in \eqref{eq:piformula} as composition of a projection along direction $z$ and three regular automorphisms $\Phi, F$ and $G$ as well as a natural embedding $e$. At each step we will keep track of points in $S\cup\{r_1,,\ldots, r_{2\nu}\}$ and vectors of $L_{a} Y$ for $ a\in S$.
		
		\textbf{Embedding $e$.} 			 
        Let $M=4N+2$. Denote the coordinates on $\R^M$ by 
                    $$ (x,y,z,w) \in \R^{N-1}\times\R\times \R \times\R^{3N+1}. $$
%                    $(x,y,z,w)$, where
%                    $$x\in\R^{N-1}, y\in\R, z\in\R, w\in\R^{3N+1}.$$

                    Let $\tilde Y:=Y\times \{0_{3N+2} \}$ be the embedding of $Y$ via $e:\R^N \to \R^N\times \{0_{3N+2} \}\subset \R^M$ and denote %$\tilde S\subset \tilde Y$ be the set
                    $$\tilde S:=\{(a,0)\in\tilde Y\,|\,a\in S\}.$$

                    Let $\tilde p$, $\tilde q$ and $\tilde r_j$ be the points $(p,0)$, $(q,0)$ and $(r_j,0)$ in $\tilde Y$, respectively.

              \textbf{Polynomial automorphism $H$.}                     
                 Define the polynomial automorphism $H:\R^M\to\R^M$ as
                $$H(x,y,z,w):=(x,y,z+h(x,y),w),$$    
                 where $h:\R^N\to\R$ is the polynomial of Lemma~\ref{lem:interpolationgradient} satisfying
                    \begin{itemize}
                        \item $h(a)=0$ for every $a\in(S\cup\{r_1,\ldots,r_{2\nu}\})\setminus\{q\}$;
                        \item $h(q)=1$;
                        \item $\nabla h(a)=0$ for every $a\in S\cup\{r_1,\ldots,r_{2\nu}\}$.
                    \end{itemize}

                 Then $H$ separates the point $\tilde q$ with all others $(\tilde S\cup\{\tilde r_1,\ldots,\tilde r_{2\nu}\})\setminus\{\tilde q\}$ to different $\R^N$-slices of $\R^M$ and the tangent directions to $H(\tilde Y)$ at points of $H(\tilde S)$ are tangent to the slices. More precisely,             
                 $$ H(\tilde a)\in\R^N\times\{0\}\times\{0\}_{3N+1} $$
                 for every $\tilde a\in(\tilde S\cup\{\tilde r_1,,\ldots,\tilde r_{2\nu}\})\setminus\{\tilde q\}$ and 
                 $$H(\tilde q)\in\R^N\times\{1\}\times\{0\}_{3N+1}.$$                
                 Moreover, if $\tilde a=(a,0)\in\tilde S\cup\{\tilde r_1,\ldots,\tilde r_{2\nu}\}$  and $\tilde v\in L_{\tilde a}\tilde Y$, then $\tilde v=(v,0)$ for some $v\in L_aY$ and
                        $$D_{\tilde a}H(\tilde v)=\tilde v=(v,0).$$

             \textbf{Regular automorphism $F$.}        
             Let $\Gamma\subset\R^N$ be any set of $|S|-1$ points contained in $\R^{N-1}\times\{0\}$ containing $q$. By Theorem~\ref{thm:JelonekAuto} there exists a regular automorphism $f_0:\R^N\to\R^N$ such that
                    \begin{itemize}
                        \item $f_0(S\setminus\{q\})=\Gamma$;
                        \item $f_0(p)=q$;
                        \item $f_0(r_j)\notin\R^{N-1}\times\{0\}$ for $j=1,\ldots,2\nu$.
                        \item The $y$ coordinate of $f_0(r_{2j-1})$ and $f_0(r_{2j})$ are different for every $j=1,\ldots,\nu$.
                    \end{itemize}
			Since $f_0$ is an automorphism, the image $D_af_0(L_aY)$ of the tangent cone to $Y$ at $a$ consists of a finite number of lines equal to the number of branches at $a$ for any $a\in Y$. In particular, we can find a rotation %around point $q$ 
			$f_1:\R^N\to\R^N$ with $f_1(q)=q$ such that
			\begin{equation}\label{eq:septan}
				D_qf_1(L_qY)\cap D_pf_0(L_pY)=\{0\}.
			\end{equation}

            Let $f$ be the regular embedding of $\R^N\times\{0,1\}\times\{0\}_{3N+1}$ onto itself defined as
                    $$f(x,y,z,w) :=\begin{cases}
                        (f_0(x,y),0,w),&\text{ if $z=0$}\\
                        (f_1(x,y),1,w),&\text{ if $z=1$}
                    \end{cases}$$
           Since $\R^N\times\{0,1\}\times\{0\}_{3N+1}$ is of dimension $N$ and $M=4N+2$, by Theorem~\ref{thm:JelonekAuto} there exists an extension of $f$ to a regular automorphism $$F:\R^M\to\R^M.$$

           Denote by 
           $$ L := \{(x,y,z,w): y=0\}$$
           the $(x,z,w)$-hyperplane of dimension $M-1$ in $\R^M$.
           
           Note that 
                    $$(F\circ H)(\tilde a)\in\R^{N-1}\times\{0\}\times\{0\}_{3N+2} \subset L \quad \text{ for }\tilde a\in \tilde S\setminus\{\tilde q\}$$
%           Moreover, 
%                    $$(F\circ H)(\tilde a)\in L$$
                    and
                    \begin{equation}\label{eq:rL}
                    (F\circ H)(\tilde r_j)\in(\R^N\times\{0\}_{3N+2})\setminus L.
                    \end{equation}

                    Moreover,
                    $$(F\circ H)(\tilde p)=(0,0,0,0) \quad \text{and}\quad (F\circ H)(\tilde q)=(0,0,1,0).$$

				    Any tangent vector $\tilde v\in L_{\tilde a}\tilde Y$ to $\tilde Y$ at $\tilde a=(a,0)\in \tilde S\cup\{\tilde r_1,\ldots,\tilde r_{2\nu}\}$ is of the form $\tilde v=(v,0)$ for some $v\in L_aY$. In particular, if $\tilde a\neq \tilde q$, then
                    $$D_{\tilde a}(F\circ H)(\tilde v)=(D_af_0(v),0)\in\R^N\times\{0\}_{3N+2},$$
                    and, if $\tilde a=\tilde q$
                    $$D_{\tilde a}(F\circ H)(\tilde v)=(D_af_1(v),0)\in\R^N\times\{0\}_{3N+2}.$$
                    
                    That is, the tangent directions to $(F\circ H)(\tilde Y)$ at points of $(F\circ H)(\tilde S\cup\{\tilde r_1,\ldots,\tilde r_{2\nu}\})$ are still tangent to the $\R^N$-slices  $\R^N\times\{0,1\}\times\{0\}_{3N+1}$.      Furthermore, from Equation~\eqref{eq:septan} follows
                    \begin{equation}\label{eq:tanpq}
                    (D_{\tilde p}(F\circ H)(L_{\tilde p}\tilde Y))\cap(D_{\tilde q}(F\circ H)(L_{\tilde q}\tilde Y))=\{0\}.
                    \end{equation}
                    
                    \textbf{Polynomial automorphism $\Phi$.} Denote
                    \begin{align*} 
                    &	Y'=(F\circ H)(\tilde Y), \\ 
                    &	S'=(F\circ H)(\tilde S), \\
                    &	p'=(F\circ H)(\tilde p)=(0,0,0,0), \\
                    &	q'=(F\circ H)(\tilde q)=(0,0,1,0), \\ 
                    &	r_j'=(F\circ H)(\tilde r_j).
                    \end{align*}
                   Let
                   $Y_\C'\subset\C^M$ be the complexification of $Y'$ and 
                   $$ L_0 := \{(x,y,z,w)\in\C^M: y=0\}$$
                   $$L_j := \{(x,y,z,w)\in\C^M: y=y(r'_j)\}$$
                   for $j=1,\ldots,2\nu$, where $y(r'_j)$ denotes the $y$ coordinate of $r'_j$. Moreover, let
                   $$B_0=(Y'_\C\setminus S')\cap L_0$$
                   $$B_j=(Y'_\C\setminus r'_j)\cap L_j$$
                   $$B=\bigcup_{j=0}^{2\nu} B_j.$$

                   For each irreducible component of $Y_\C'$ there is some smooth point $r'_j$ in it and by~\eqref{eq:rL} $r'_j$ does not belong to the hyperplane $L_0$. Thus each irreducible component of the curve $Y'$ is not contained in the hyperplane $L_0$. Similarly,  there is no irreducible component contained in any $L_j$, since the set $\{r_1,\ldots,r_{2\nu}\}$ contains two points of each irreducible component and $y(r'_{2j-1})\neq y(r'_{2j})$ for every $j=1,\ldots,\nu$. Consequently, the set $B$ is finite.

                   Consider the projections
                    \begin{gather*} 
                    	\pi_{(x,y)}:\C^M\to\C^N \\
                    	(x,y,z,w)\to(x,y)
                    \end{gather*}
                     and
                    \begin{gather*} 
                    \pi_{w_1}:\C^M\to\C \\	
                    (x,y,z,w_1,\ldots,w_{3N+1})\to w_1.
                    \end{gather*}

                   Let $\varphi:\C^N\to\C$ be a polynomial with real coefficients such that
                   $$\varphi(\pi_{(x,y)}(S'))=\varphi(\pi_{(x,y)}(r'_j))=0$$
                   and
                   $$\varphi(\pi_{(x,y)}(b))+\pi_{w_1}(b)\neq0$$
                   for every $b\in B$.

                   Define the polynomial automorphism $\Phi:\C^M\to\C^M$ as
                    $$\Phi(x,y,z,w)=(x,y,z,w_1+\varphi(x,y),w_2,\ldots,w_{3N+1}).$$

                    Then $\Phi$  fixes the points of $S'$ as well as each $r'_j$, $j=1,\ldots,2\nu$. Moreover, for any point $a'\in Y_\C'\setminus S'$, we have either
                    $$\Phi(a')\notin\{(x,y,z,w)\in\C^M\,|\,w=0\} \text{ when } a'\in B$$
                    or
                    $$\Phi(a')\notin L_0  \text{ when } a'\notin B,$$
					because in the second case $a'\notin L_0$. 
                    Hence the only points of the curve $\Phi(Y'_\C)$ lying on the $(N+1)$-dimensional linear space
                    $$K\coloneqq\{(x,y,z,w)\in\C^M\,|\,y=0,w=0\}$$
                    are the points of $S'$, i.e.,
                    \begin{equation}\label{eq:RSpoints}
                        \Phi(Y_\C')\cap K=\Phi(S')=S'.
                    \end{equation}

                    Similarly, for each $j=1,\ldots,2\nu$, the only point of $\Phi(Y'_\C)$ lying on
                    $$K_j\coloneqq\{(x,y,z,w)\in\C^M\,|\,y=y(r'_j),w=0\}$$
                    is $r'_j$, i.e.,
                    \begin{equation}\label{eq:rpoints}
                        \Phi(Y_\C')\cap K_j=\Phi(r'_j)=r'_j.
                    \end{equation}

                    Additionally, if $a'\in S'\cup\{r'_1,\ldots,r'_{2\nu}\}$ and $v'\in L_{a'}Y'$, then $v'=(v,0)$ for some $v\in \R^N$. Consequently                    
                    \begin{equation}\label{eq:phitan}
                    D_{a'}\Phi(v)=\left(v,0,\sum_{j=1}^{N-1}\frac{\partial\varphi}{\partial x_j}(a')v_j,0,\ldots,0\right).
                    \end{equation}

                    \textbf{Projection.} Denote
                    \begin{align*}
                    &\hat{Y}=\Phi(Y'),\\
                    &\hat{S}=\Phi(S'), \\
                    &\hat{p}=\Phi(p')=(0,0,0,0), \\
                    &\hat{q}=\Phi(q')=(0,0,1,0).
                    \end{align*}

                    Consider the projection along the $z$ direction
                    \begin{gather*}
                    	\hat{\pi}:\R^M\to\R^{M-1}\\
                    	\hat{\pi}(x,y,z,w)=(x,y,w).
                    \end{gather*}
                    
                    The center of this projection is a real point which does not belong to $\hat{Y}_\C$. Hence it is a well-defined projection on the complexification.

		\textbf{Mapping $\pi$.}  
%		Let $\phi:Y\to\R^{M}$  be given by the composition
%                    \begin{equation*}
%                        \phi:=.
%                    \end{equation*}
%                    
         Define $\pi:Y\to\R^{N'}$ with $N':=4N+1$ as given by the composition
         \begin{equation}\label{eq:piformula}
         	\pi:=\hat\pi\circ\Phi\circ F\circ H\circ e
         \end{equation}
          
        Note that $\phi:=\Phi\circ F\circ H$ is a regular automorphism of $\R^M$. Moreover, $\phi(Y) = \hat{Y}$ is a real curve, since $\phi$ is regular with real coefficients.

                    \textbf{Claim (1):}  Let $Y_\C\subset\C^M$ be the complexification of $\hat Y$.      
                    The map
                    $$\pi_\C=\hat\pi_\C|_{Y_\C}:Y_\C\to\pi_\C(Y_\C)$$
                    is birational. In fact, each irreducible component $Y_j$ of $Y_\C$ has at least one smooth point, for instance $r_{2j-1}$, such that                                  
                     $(\pi_\C|_{Y_\C})^{-1}(\pi_\C|_{Y_\C}(r_{2j-1}))=r_{2j-1}$  and $\pi_\C|_{Y_\C}$ is a local isomorphism at this point.  Thus 
                    $\pi_\C|_{Y_\C}$ is generically injective and transversal to $Y_\C$ on each of its irreducible components, according to Equations~\eqref{eq:rpoints} and~\eqref{eq:phitan}, respectively.
                    
                    Consequently, if
                    $$Z\coloneqq Re(\pi_\C(Y_\C))$$
                    is the real locus of the image then
                    $$\pi:Y\to Z$$
                    is birational. Since $Y$ is compact,   $\pi$ is a proper map onto its image.

                    \textbf{Claim (2):} By definition of $\pi$, the set $Z\setminus\pi(Y)$ consists of points $c\in (x,y,w)\in\R^{M-1}$ such that $d=(x,y,z,w)\in Y_\C\setminus \hat{Y}$ for some $z\notin\R$. Note that if 
                    $\hat{\pi}(d)=c$, then also $\hat{\pi}(\overline{d})=c$ for the conjugate.
                    Since $\pi$ is birational  (see Claim (1)) we get that the set of such points $c$ has to be finite:
                    \begin{equation}\label{eq:isopoints}
                    	Z=\pi(Y)\cup\{c_1,\ldots,c_k\}.
                    \end{equation}              
                    
                    \textbf{Claim (3):} We verify                    
                    $$\pi(p)=\hat\pi(\hat{p})=(0,0,0)=\hat\pi(\hat{q})=\pi(q).$$
                    Since the fiber $\hat{\pi}^{-1}(0,0,0)$ is contained in $K=\{y=w=0\}$, Equation~\eqref{eq:RSpoints} implies
                    $$\hat\pi^{-1}(0,0,0)=\{\hat{p},\hat{q}\}.$$
                    Consequently
                    $$\pi^{-1}(0,0,0)=\{p,q\}.$$
                    Similarly, $\hat{\pi}$ is one-to-one on $\hat{S}\setminus\{\hat{p},\hat{q}\}$, because all points of                     
                    $\hat{S}\setminus\{\hat{p},\hat{q}\}$ lie in $ K$ and they share the same $y,z$ and $w$ coordinates, thus they have different $x$ coordinates. Consequently $\pi$ is one-to-one on $S\setminus\{p,q\}$.

                    \textbf{Claim (4):} If $a\in S\setminus\{p,q\}$ is a non-singular point of $Y$, then by Claims (1) and (3) its image by $\pi$ is a non-singular point of the real curve $Z$. 
                    
                    Let $a\in S\setminus\{p,q\}$ be a singular point of $Y$. Denote $\hat{a}=(\Phi\circ F\circ H)(a)$. Lemma \ref{lem:transvcone} implies that the image $\hat{Y}$ by a regular automorphism has a transverse intersection at $\hat{a}$. 
                    From Equation~\eqref{eq:phitan} we have that the projection $\hat\pi$ is one-to-one on $L_{\hat{a}}\hat{Y}$. Since $D_{\hat{a}}\hat{\pi}=\hat{\pi}$, then $\pi(Y)$ has a transverse intersection at $\pi(a)$.

                    To see that this is also true at $\pi(p)$, note that Equations~\eqref{eq:phitan} and~\eqref{eq:tanpq} imply that $\hat\pi$ is one-to-one on
                    $$L_{\hat{p}}\hat{Y}\cup L_{\hat{q}}\hat{Y},$$
					where %by Lemma~\ref{lem:transvcone} and construction of $F$ and $\Phi$ we have
					$$L_{\hat{p}}\hat{Y}\cap L_{\hat{q}}\hat{Y} = \{0\}.$$
                    Since $\hat{Y}$ has transverse intersections at $\hat{p}$ and $\hat{q}$, the conclusion follows.
				\end{proof}
	Note that in the target curve in Lemma~\ref{lem:transintersections} at the moment of projecting we may have introduced new, non-transverse-intersection singularities at points outside the control set. Therefore, we need to refine Lemma~\ref{lem:transintersections} by sending isolated points to infinity and embeddedly resolving the undesired singularities. This is the content of Lemma~\ref{lem:transonly} below. 
			
            \begin{lem}[Refined pointwise transverse glueing]\label{lem:transonly}
					
                    Let $Y$ be a  compact algebraic curve without isolated points in $\R^N$ such that all its singular points %$p_1,\dots,p_{i}$ 
					are  transverse intersections of smooth branches. For any $p,q\in Y$ there exists a mapping $\rho:Y\to W$, where $W:=\rho(Y)\subset\R^{N'}$ is a compact algebraic curve, such that
					\begin{enumerate}
						\item $\rho$ is birational
						\item $W$ has no isolated points
						\item $\rho^{-1}(\rho(p))$ consists of two points $p$ and $q$
						%\item curve $\pi(Y)$ is transverse intersection at $\pi(p_i)$
						\item {$\Sing(\rho(Y)) = \rho(\Sing(Y)) \cup \rho(p)$}
						\item $\rho$ is one-to-one outside $\{p,q\}$ %on $Y\setminus (\Sing(Y) \cup \rho^{-1}(p))$% for all $p_1,\dots,p_{i-1}$
						\item $W$ has only smooth points or  singularities which are transverse intersections of smooth branches.
					\end{enumerate}				
				\end{lem}
                
				\begin{proof}
                Let $S=\Sing(Y)\cup\{p,q\}$ and $\pi$ be the mapping satisfying claims of  Lemma~\ref{lem:transintersections}. Under the notation of that Lemma, the mapping $\pi$ is a bi-rational relation between a compact connected real algebraic curve $Y$ and a compact real algebraic curve $Z$ so that
                $$Z=\pi(Y)\cup\{c_1,\ldots,c_k\}\subset\R^{M-1}$$
                with $c_1,\ldots,c_k$ isolated points of $Z$.

                Let $g:\R^{M-1}\to\R$ be a regular defined as
                $$g(a):=\prod_{i=1}^k||a-c_i||^2.$$
                Note that $g(c_i)=0$ for every $i=1,\ldots, k$ and $g(a)\neq0$ for every $a\in\pi(Y)$. %Moreover, $g$ is smooth outside of points $\{c_1,\ldots,c_k\}$. 
                 Consider the compact connected real algebraic set
                $$\tilde Z=\{(a,t)\in\R^M\,|\, P_1(a)=\dots=P_\xi(a)=0,g(a)t=1\}\subset\R^M$$
                where $P_1,\dots, P_\xi$ are the defining equations of $Z$.

                The mapping
                \begin{gather*}
                	\alpha:Z\to\tilde Z, \quad 
                    a\to\left(a,\ {1}/{g(a)}\right)
                \end{gather*}
                is birational. At every point $\alpha(\pi(a))$ for $a\in S$ the curve $\tilde Z$ is a transverse intersection, since the sets $\pi(S)$ and $\{c_1,\ldots,c_k\}$ are finite and disjoint.                
                Moreover, the composition of birational mappings $\alpha\circ\pi:Y\to\tilde Z$ is a well-defined function  onto the compact connected real algebraic curve  $\tilde Z$ in $\R^M$ which is  1-1 on  $S\setminus \{p,q\}$.

                Now consider the set of points on the image which do not preserve local inner Lipschitz properties of $Y$, which is
                $$R\coloneqq \pi(Y)\cap \Sing(Z)\setminus\pi(S) . $$%\subset\pi(Y).$$
                Denote $\tilde R:= \alpha(R) \subset\tilde Z$.

                Let
                $$Z'\subset\Res_{\R^M}(\tilde Z,\tilde R)$$
                be the strict transform of $\tilde Z$ under the real embedded resolution
                $$\sigma: (\Res_{\R^M}(\tilde Z,\tilde R), Z'\cup E, E) \to (\R^N, \tilde Z, \tilde R)$$
                of $\tilde Z\subset\R^M$ at points of $\tilde R$, see for instance~\cite{BierstoneMilman, Wlodarczyk2005}.        
                Note that $Z'$ is a connected compact algebraic curve smooth outside of the $|S|-1$ points of $\sigma^{-1}(\alpha(\pi(S)))$. Since the embedded resolution outside of exceptional locus is a local isomorphism, thus in points of $\sigma^{-1}(\alpha(\pi(S)))$ the curve~$Z'$ has transverse intersections.
                
                %Reassuming, the mapping $\sigma_{|Z'\setminus E}$ is 1-1 onto $\tilde{Z}\setminus \tilde R$. 
                By properties (1), (3) and (4) of Lemma~\ref{lem:transintersections} the birational mapping $$\sigma^{-1}\circ \alpha\circ \pi:Y\to Z'$$
                is a well-defined function from $Y$ onto $Z'$ which is everywhere injective except for points $p,q$ where it takes the same value. 
               
                Let $\beta:\Res_{\R^M}(\tilde Z,\tilde R)\to \R^{N'}$ be a real biregular embedding of the resolution space into $\R^{N'}$ as a compact algebraic set, compare~\cite{BCR}. 
     
     Denote $W:= \beta(Z')$.        
    Define $$ \rho := \beta\circ \sigma^{-1}\circ \alpha\circ \pi : Y\to \R^{N'}. $$  
    The mapping is birational between  $Y$ and $W$. Moreover it is 1-1 between $Y\setminus\{p,q\}$ and $W\setminus \rho(p)$ with $\rho(p)=\rho(q)$.  In particular, the real algebraic curve $W$ has no isolated points.   Since $Z'$ had only transverse intersections at $\sigma^{-1}(\alpha(\pi(S)))$ of the same type as $Z$ of Lemma~\ref{lem:transintersections}, the curve $W$ at every point $a\in W\setminus \{p,q\}$ is either smooth or has transverse intersections of the same type as $Z$. Moreover, at the point $\rho(p)=\rho(q)$ it is also a transverse intersection.
\end{proof}

            \begin{lem}[Inversion of an algebraic set is real algebraic]\label{lemma:inversion}
            	Let $X\subset\R^n$ be a real algebraic set and $I(x)=\dfrac{x}{\|x\|^2}$ the inversion map of $\R^n$. Then the closure 
            	$\overline{I(X\setminus\{0\})}$ 
            	is a real algebraic set {birational to $X$}.
            \end{lem}
            
            \begin{proof}
            	Let $X=V(f_1,\dots, f_k)$ with $ f_1, \dots, f_k \in \mathbb{R}[x_1,\dots,x_n]$. Note that the inversion outside the origin is its own inverse {and it is birational}. 
            	If $\deg f_i = d_i$, define 
            	$$g_i(y) := \|y\|^{2d_i} f_i\left(\frac{y}{\|y\|^2}\right).$$
            	Then each $g_i$ is a polynomial and for $y\neq 0$ we have
            	$$f_i\left(\frac{y}{\|y\|^2}\right)=0 \;\Longleftrightarrow\; g_i(y)=0  . $$
            	Therefore,
            	$$I(X\setminus\{0\}) = \{ y \in \R^n\setminus\{0\} : g_1(y)=\cdots=g_k(y)=0 \}.$$
            	
            	Let $Y:= V(g_1,\dots, g_k)$. 
            	%$$Y := \{ y \in \R^n : g_1(y)=\cdots=g_k(y)=0 \}.$$
            	{If $X$ is unbounded, then $Y = I(X\setminus\{0\}) \cup \{0\}$ and if $X$ is compact, then $Y = I(X\setminus\{0\})$.} This gives the claim.
            \end{proof}

				\begin{proof}[Proof of Theorem ~\ref{thm:algebraicModels}.] \label{proof:thm}
				Let $X \subset \R^n$ be a real algebraic curve and $0\notin X$. Let $X'$ be the curve itself if $X$ is compact or the closure of the inversion of $X$ as in Lemma~\ref{lemma:inversion}. Then $X'$ is a compact algebraic curve such that $X'\setminus \{0\}$ is inner bi-Lipschitz with $X$.
				
				Similarly as in proof of Lemma~\ref{lem:transonly} consider the embedded resolution of $X'$ to a smooth curve~$\tilde X$ by
				$$\sigma: (\Res_{\R^n}(X',\Sing(X)), \tilde X\cup E, E  )  \to (\R^n, X', \Sing(X))$$
				and the biregular embedding
				$$\beta : \Res_{\R^n}(X',\Sing(X)) \to \R^N$$
				of the resolution space as a compact algebraic set.
				
				Then $Y:= \beta(\tilde X)\subset \R^N $ is a smooth compact algebraic curve. Consider
			% the set $ \beta(\tilde X \cap E) \subset Y$ which is mapped by 
			the well-defined proper function $$\mu:=\sigma_{|\tilde X} \circ \beta^{-1}: Y\to X'.$$ 
			
			Let 
			$$ S':=\{ s\in \Sing(X'): \ s \text{ is a self-intersection point of } X' \}. $$ 
			By Lemma~\ref{lem:curvesLNEBALLS}, % and Remark~\ref{rk:sm}, 
			germs of curves $(X', a)$ and $(Y, \mu^{-1}(a))$ are locally inner-bi-Lipschitz for every $a\in \mu^{-1}(\Sing(X')\setminus S')$, in particular the germ of mapping $\mu$ near the unique point $\mu^{-1}(a)$ gives the equivalence.
			
			Now, let 
			$$S:= \mu^{-1}(S')$$
			and apply Lemma~\ref{lem:transonly}  inductively to glue together all points of each fiber $\mu^{-1}(s)$ for $s\in S'$ while taking care of not glueing other points. More precisely, we obtain a birational map
			$$ \Pi: Y\to W $$
			such that $\Pi$ is a well-defined proper map on $Y$ with image equal $W$, where $W$ is a compact connected algebraic curve in some $\R^N$ with transverse intersection at every unique point $\Pi(\mu^{-1}(s))$ for $s\in S'$ and smooth otherwise. 
			
			Therefore, the map
			$$ \Pi\circ \mu^{-1} : X'\to W $$
			is an injective proper map which is a restriction of a birational map. %  Thus irreducible components of $X'$ are mapped to irreducible components of $W$. 
			More importantly, the algebraic curve $W$ is an LNE model of $X'$.
			
		If $X$ was compact, we get the claim. If $X$ was not compact, let $p_\infty:= (\Pi\circ \mu^{-1})(0) \in W$. Without loss of generality, up to affine translation, we can assume $p_\infty = 0$.  
			Applying  inversion onto the compact  set $W$ by Theorem~\ref{thm:CGMLNE} and Lemma~\ref{lemma:inversion} we get the first claim. 

            Moreover, a generic projection of a real algebraic curve to $\R^3$ is outer bi-Lipschitz by \cite{LevAlexJelonekEmbeddings}. Hence we can assume that $W$ is in $\mathbb{R}^3.$
		\end{proof}

\begin{rk}
	One can compute equations of the real algebraic LNE model in Theorem \ref{thm:algebraicModels} explicitly from initial equations of the real algebraic curve.
\end{rk}

Note that in Lemmas \ref{lemma:inversion}, \ref{lem:transonly} and \ref{lem:transintersections} we do not need connectedness, thus we obtain the following corollary.

\begin{cor}\label{cor:classification}
	Inner Lipschitz classification of real algebraic curves is equivalent to outer Lipschitz classification in $\R^3$ of real algebraic curves with LNE connected components.
\end{cor}

			\section{LNE models of complex algebraic curves}

            Now we discuss LNE models of complex algebraic curves. We see in Theorem~\ref{thm:noComplexmodels} that  complex algebraic curves may not admit complex algebraic LNE models. This is unlike in the real setting discussed in previous Section~\ref{sec:algebraicModels}. %The following theorem is proved in \cite{LevAlexJelonekEmbeddings}[Theorem  4.8].

			\begin{thm}\label{thm:noComplexmodels}
				Complex algebraic LNE model of a smooth connected complex algebraic curve need not exist.
				%There exist complex algebraic curves without algebraic LNE models.
			\end{thm}
			
			\begin{proof}
				Let $\tilde X\subset\C\Ps^n$ be a smooth complex algebraic curve of genus $g\geq 1$ and consider a point $P\in\tilde X$ and a divisor $mP$ such that $m\geq 2g+1$. This divisor is very ample and defines an embedding of $\tilde X$ into a projective space such that there exists a hyperplane $H_\infty$ intersecting $\tilde X$ only at $P$. Removing $H_\infty$ yields a smooth affine algebraic curve $X$ of genus $g$ with a single end at infinity.
				
				If $X$ has a complex  algebraic LNE model $Y$, then $Y$ also has genus $g$ and one end at infinity. On the other hand, by \cite[Theorem 8.1]{MVAlnecurves}, the number of points on intersection of the projective closure of $Y$ with the hyperplane at infinity is the degree of $Y$. Thus, $\deg Y=1$ and $Y$ is a complex line, which contradicts $g\geq1$.
			\end{proof}

				This still leaves  open a very interesting question whether  affine complex algebraic curves admit real algebraic LNE models.

			\bibliographystyle{alpha}
			\bibliography{AlgebraicModels}
			
		\end{document}